\documentclass[12pt]{amsart}
\usepackage{geometry}   
\usepackage[colorlinks,citecolor = red, linkcolor=blue,hyperindex]{hyperref}
\usepackage{euscript,eufrak,verbatim, mathrsfs}
\usepackage[psamsfonts]{amssymb}
\usepackage{bbm}
\usepackage{graphicx}
 \usepackage{float}
\usepackage{float, tikz}

 \usepackage{extarrows}
 \usepackage[all, cmtip]{xy}

\usepackage{upref, xcolor, dsfont}
\usepackage{amsfonts,amsmath,amstext,amsbsy, amsopn,amsthm}
\usepackage{enumerate}

\usepackage{url}

\usepackage{mathtools}
\usepackage{bookmark}

 \usepackage{euscript}
\usepackage{helvet}         
\usepackage{courier}        
\usepackage{type1cm}        
\usepackage{multicol}        
\usepackage[bottom]{footmisc}

\newtheorem{theorem}{Theorem}[section]
\newtheorem*{theorem*}{Theorem A}
\newtheorem{lemma}[theorem]{Lemma}

\newtheorem{proposition}[theorem]{Proposition}
\newtheorem{corollary}[theorem]{Corollary}
\newtheorem*{definition*}{Definition}
\newtheorem*{remark*}{Remark}

\newtheorem*{observation*}{Observation}

\newtheorem*{assumption*}{Assumption}

\theoremstyle{definition}

\newtheorem{question}{Question}

\theoremstyle{remark}
\newtheorem{remark}{Remark}[section]

\newcommand{\C}{\mathbb{C}}

\newcommand{\supp}{\mathrm{supp}}

\newcommand{\Hol}{\mathrm{Hol}}
\newcommand{\Harm}{\mathrm{Harm}}

\newcommand{\an}{\text{\, and \,}}

\begin{document}

\title[Bergman spaces and bounded point evaluations]{Complete weighted Bergman spaces have bounded point evaluations}

\author
{Yong Han}
\address
{Yong HAN: College of Mathematics and Statistics, Shenzhen University, Shenzhen 518060, Guangdong, China}
\email{hanyongprobability@gmail.com}

\author
{Yanqi Qiu}
\address
{Yanqi QIU: School of Mathematics and Statistics, Wuhan University, Wuhan 430072, Hubei, China;  Institute of Mathematics, AMSS, Chinese Academy of Sciences, Beijing 100190, China}
\email{yanqi.qiu@hotmail.com}

\author{Zipeng Wang}
\address{Zipeng WANG: College of Mathematics and Statistics, Chongqing University, Chongqing
401331, China}
\email{zipengwang@cqu.edu.cn}

\thanks{Y. Han is supported by the grant NSFC 11688101.  Y. Qiu is supported by grants NSFC Y7116335K1,  NSFC 11801547 and NSFC 11688101. Z. Wang is supported by the grant NSFC 11601296}

\begin{abstract}
Let $\Omega\subset \mathbb{C}$ be an arbitrary domain in the  one-dimensional complex plane equipped with a  positive Radon measure $\mu$.  For any $1\le p< \infty$, it is  shown that the weighted Bergman space  $A^p(\Omega, \mu)$ of holomorphic functions is a Banach space if and only if  $A^p(\Omega, \mu)$  has locally uniformly bounded  point evaluations.  In particular, in the case $p =2$,   any complete Bergman space $A^2(\Omega, \mu)$ is automatically a reproducing kernel Hilbert space.
\end{abstract}

\subjclass[2020]{Primary 30H20,46A30; Secondary 46A30, 46E22}
\keywords{weighted Bergman spaces; bounded point evaluations; reproducing kernel Hilbert spaces}

\maketitle

\setcounter{equation}{0}

\section{The main result}

Let  $\Omega\subset \C$ be an arbitrary domain (i.e., non-empty open connected subset) in the one-dimensional complex plane equipped with a positive Radon measure $\mu$.  For any $1\le p <\infty$, define the Bergman space by
\begin{align}\label{def-berg}
A^p(\Omega, \mu): = \Big\{f \in \Hol(\Omega)\Big| \int_\Omega |f|^p d\mu<\infty \Big\},
\end{align}
where as usual,  $\Hol(\Omega)$ denotes the set of all  holomorphic functions on $\Omega$.   For any $z\in \Omega$, the point evaluation $T_z$ on $A^p(\Omega, \mu)$ is the linear map defined by
\[
  A^p(\Omega, \mu) \ni f \mapsto T_z(f)= f(z).
\]

The space  $A^p(\Omega, \mu)$  is equipped with the natural semi-norm
\begin{align}\label{def-norm}
\| f\|_{A^p(\Omega, \mu)}: = \Big(\int_\Omega |f|^p d\mu\Big)^{1/p}, \quad \forall f \in A^p(\Omega, \mu).
\end{align}
The above semi-norm is a norm  on $A^p(\Omega, \mu)$ if and only if $\Omega$ contains an accumulation point of  the support $\supp(\mu)$ of the measure $\mu$. Indeed,  if $\supp(\mu)$ has no accumulation points in $\Omega$, then the semi-norm  defined in \eqref{def-norm} is not a norm since  there would be a function  $f\in \Hol(\Omega)$ such that $f \not\equiv 0$ but  $f|_{\supp(\mu)} = 0$ and hence $\| f\|_{A^p(\Omega,\mu)}=0$ (for the existence of such functions, see Rudin \cite[Thm. 15.11]{Real-complex-analysis}). Conversely, if $\Omega$ contains an accumulation point of $\supp(\mu)$, then \eqref{def-norm} clearly defines a norm, see Lemma~\ref{lem-mono} below for more details.

The main result of this note is

\begin{theorem}\label{thm-main}
Let  $\Omega\subset \C$ be an arbitrary domain equipped with a positive Radon measure $\mu$ and let $1\le p<\infty$. If $A^p(\Omega, \mu)$ is Banach space, then  for any compact subset $K\subset \Omega$, there exists a constant $\gamma_K> 0$ such that
\begin{align}\label{bdd-pe}
\sup_{z\in K}|f(z)| \le \gamma_K \| f\|_{A^p(\Omega, \mu)}, \, \forall f \in A^p(\Omega, \mu).
\end{align}
\end{theorem}

If \eqref{bdd-pe} holds for any compact subset $K\subset\Omega$, then $A^p(\Omega,\mu)$ is said to have locally uniformly bounded point evaluations.

\begin{remark}
If $A^p(\Omega,\mu)$ has locally uniformly bounded point evaluations, then by a routine argument,  one proves that $A^p(\Omega, \mu)$ is  a Banach space.   Hence Theorem~\ref{thm-main} in fact implies that  $A^p(\Omega, \mu)$ is a Banach space if and only if it has locally uniformly bounded point evaluations.
\end{remark}

\begin{remark}\label{rem-0p1}
The assumption $1\le p<\infty$ in Theorem \ref{thm-main}  is not essential. Replacing the  classical Closed Graph Theorem for Banach spaces in the proof below of Theorem \ref{thm-main} by the one for $F$-spaces (see  \cite[Section 1.4, Cor. 1.7]{kalton}), one immediately obtains the following analogous result for quasi-Banach spaces:  let  $\Omega\subset \C$ be an arbitrary domain equipped with a positive Radon measure $\mu$ and let $0< p< \infty$, then  the weighted Bergman space $A^p(\Omega, \mu)$ is a quasi-Banach space if and only if it has locally uniformly bounded point evaluations.
\end{remark}

\begin{remark}
The vector-valued analogue of Theorem \ref{thm-main} can be proved in the same way.  Let  $\Omega\subset \C$ be an arbitrary domain equipped with a positive Radon measure $\mu$ and let $1\le p<\infty$.   Let $X$ be a Banach space over $\C$. The vector-valued Bergman space $A^p(\Omega, \mu; X)$ is defined as
\[
A^p(\Omega, \mu; X): = \Big\{f: \Omega\rightarrow X\Big| \text{$f$ is holomorphic and\,}\int_\Omega \|f(z)\|_X^p d\mu(z)<\infty \Big\},
\]
equipped with the semi-norm
\[
\| f\|_{A^p(\Omega, \mu; X)}: = \Big(\int_\Omega \|f(z)\|^p d\mu(z)\Big)^{1/p}, \quad \forall f \in A^p(\Omega, \mu; X).
\]  Then  $A^p(\Omega, \mu; X)$ is Banach space if and only if  for any compact subset $K\subset \Omega$, there exists a constant $\gamma_K> 0$ such that
\[
\sup_{z\in K}\|f(z)\|_X \le \gamma_K \| f\|_{A^p(\Omega, \mu; X)}, \, \forall f \in A^p(\Omega, \mu; X).
\]
\end{remark}

\subsection{Historical remarks}

Determining the completeness of weighted Bergman spaces $A^p(\Omega,\mu) (1\le  p< \infty)$ for  a given  Radon measure $\mu$ on a planar domain $\Omega$ is a basic question in the theory of holomorphic function spaces and related topics.   This question is in general very difficult and the situation can be rather complicated.  For instance,  for any bounded domain $\Omega$, there exists a measure $\mu$ on $\Omega$ such that the Bergman space $A^p(\Omega,\mu)$ is dense in $L^p(\Omega,\mu)$ for  $1\leq p<\infty$ (see \cite[Thm. 6]{Bram}). 

 Under the assumption that $A^p(\Omega, \mu)$ is a Banach space and  the point evaluations are  bounded for all points in $\Omega$, then by the Banach-Steinhaus Theorem, the set where the point evaluations are bounded is closed and moreover,  the point evaluations are locally uniformly bounded (as was done in  Arcozzi and Bj\"{o}rn \cite[Prop. 5.8 and Prop. 5.10]{Dominating-sets}). The novelty of Theorem~\ref{thm-main} is the everywhere boundedness of the point evaluations.  

For more discussions about the bounded point evaluations on weighted Bergman spaces and their applications to operator theory, the reader is referred to \cite[\S 2 of Ch. VIII]{conway}, \cite{thomson} and \cite{ASS}.

Let $\Sigma$ be a set and $\mathbb{F}$ be either $\mathbb{R}$ or $\mathbb{C}$. A set $\mathcal{H}$ consisting of bona fide $\mathbb{F}$-valued functions on $\Sigma$ (not the equivalence classes of functions such as in the usual case of $L^2(\Sigma, \nu)$ for some measure $\nu$) is called a reproducing kernel
Hilbert space on $\Sigma$ (RKHS on $\Sigma$) (see, e.g., \cite[pp. 19-20]{halmos} and \cite[Def. 1.1]{paulsen}) if
\begin{itemize}
\item[(1)] $\mathcal{H}$ is a $\mathbb{F}$-vector space endowed with the natural operations of addition and scalar multiplication;
\item[(2)] $\mathcal{H}$ is a Hilbert space;
\item[(3)] all point evaluations on $\mathcal{H}$ are bounded.
\end{itemize}
It is known that there exist Hilbert spaces of functions with  at least one unbounded point evaluation (see  \cite[the solution of Problem 36, pp. 192]{halmos}).
However, Theorem~\ref{thm-main} implies that, in sharp contrast to the definition in  standard textbooks,  in the case of weighted Bergman spaces $A^2(\Omega,\mu)$,   the item (3) in the definition of RHKS is superfluous, since it follows from the completeness of $A^2(\Omega,\mu)$ in the item (2).

\begin{corollary}
Let  $\Omega\subset \C$ be an arbitrary planar domain equipped with a positive Radon measure $\mu$. If $A^2(\Omega, \mu)$ is a Hilbert space, then it is a reproducing kernel Hilbert space.
\end{corollary}

\subsection{Questions}

\begin{question}
Let  $\Omega\subset \C$ be a planar domain equipped with a positive Radon measure $\mu$ and let $1\le p<\infty$. Assume that all point evaluations are bounded on $A^p(\Omega, \mu)$, does it follow that the point evaluations are locally uniformly bounded  (or equivalently, does it follow that $A^p(\Omega, \mu)$ is a Banach space) ? 
\end{question}

\begin{remark}
Thomson  \cite{thomson} proved  the following striking result: let $\mu$ be a  positive Radon measure on $\C$ with compact support.  Let $P^2(\mu)$ be the $L^2(\mu)$-closure of the space of complex-coefficient polynomials.   Then under the assumption  of the purity (see  Conway \cite[\S 2 of Ch.II, pp. 38]{conway} for its precise definition) of the operator $S_\mu: P^2(\mu)\rightarrow P^2(\mu)$ defined by $S_\mu (f) (z):= zf(z)$, the boundedness of the point evaluation at a point $a\in \C$ implies the uniform boundedness of the point evaluations in a neighborhood of $a$. 
\end{remark}

The weighted harmonic Bergman space  $B^p(\Omega, \mu)$ is defined by
$$
B^p(\Omega, \mu): = \Big\{f \in \Harm(\Omega)\Big| \int_\Omega |f|^p d\mu<\infty \Big\},
$$
where $\Harm(\Omega)$ denotes the set of all  harmonic functions on $\Omega$.

\begin{question}
Let  $\Omega\subset \C$ be an arbitrary domain equipped with a positive Radon measure $\mu$ and let $1\le p<\infty$.
If $B^p(\Omega, \mu)$ is a Banach space, does it follow that $B^p(\Omega, \mu)$ has locally uniformly bounded point evaluations ?
\end{question}

The following partial result is obtained. 
\begin{proposition}\label{prop-harmonic}
Let  $\Omega\subset \C$ be an arbitrary domain equipped with a positive Radon measure $\mu$ and let $1\le p<\infty$.
Assume that  the interior of the set $\supp(\mu)$ is non-empty. If $B^p(\Omega, \mu)$ is a Banach space, then it has locally uniformly bounded point evaluations. 
\end{proposition}

For domains in higher dimensional complex Euclidean space $\C^d$, the weighted Bergman spaces are defined similarly as in
\eqref{def-berg}.

\begin{question}
Fix any $d\ge 2$ and let  $\Omega\subset \C^d$ be an arbitrary domain equipped with a positive Radon measure $\mu$. If $A^p(\Omega, \mu)$ is a Banach space, does it follow that $A^p(\Omega, \mu)$ has  locally uniformly bounded point evaluations ?
\end{question}

Let  $\Omega\subset \C^d$ be a domain.  We say that a closed subset $E\subset \Omega$ has {\it interior $H^\infty$-uniqueness property} with respect to $\Omega$ if there exists a  relatively compact sub-domain  $S\subset \Omega$ such that  $S\cap E$ is non-empty and  any bounded holomorphic function $f\in \Hol(S)$ with   $f|_{S\cap E} \equiv 0$ must be identically zero on $S$.

\begin{proposition}\label{prop-higher}
Fix any $d\ge 2$ and let  $\Omega\subset \C^d$ be an arbitrary domain equipped with a positive Radon measure $\mu$ such that $\supp(\mu)$ has interior $H^\infty$-uniqueness property with respect to $\Omega$.   If $A^p(\Omega, \mu)$ is a Banach space, then it has locally uniformly bounded point evaluations. 
\end{proposition}

The proofs of both Proposition~\ref{prop-harmonic} and Proposition~\ref{prop-higher} are similar to that of Theorem~\ref{thm-main} and thus are omitted.

\section{The proof of Theorem \ref{thm-main}}
Our proof of Theorem \ref{thm-main} is self-contained. The  idea is  rather elementary:  suppose that $A^p(\Omega, \mu)$ is a Banach space, then
\begin{itemize}
\item[(1)] The set $\Omega$ contains an accumulation point of $\supp(\mu)$ since  the semi-norm in \eqref{def-norm} is a norm.
\item[(2)] For any {\it relatively compact} sub-domain $S\subset \Omega$ containing an accumulation point of $\supp(\mu)$, the point evalutions on $A^p(\Omega, \mu)$ are uniformly bounded on $S$ (see Lemma \ref{lem-S} for details). More precisely,  there exists a constant $\gamma_S> 0$ such that
\[
\sup_{z\in S} |f(z)| \le \gamma_S \| f\|_{A^p(\Omega, \mu)}.
\]
\item[(3)]The proof of Theorem \ref{thm-main} is then completed since any compact subset $K \subset \Omega$ is contained in a relatively compact sub-domain $S\subset \Omega$ containing an accumulation point of $\supp(\mu)$.
\end{itemize}

\begin{lemma}\label{lem-supp-baohan}
Let $U\subset \Omega$ be an open subset of $\Omega$. 
 If $f\in \Hol(U)$ satisfies $f(z)=0$ for $\mu$-almost every $z\in U$, then $f|_{\mathrm{supp}(\mu) \cap U}  \equiv 0$. 
\end{lemma}
\begin{proof} 
Let $f \in \Hol(U)$ be a holomorphic function vanishing $\mu$-almost everywhere in $U$. Then  the open subset $\{z\in U: f(z) \ne 0\}$ satisfies 
\[
\mu \big(\{z\in U: f(z) \ne 0\}\big)=0.
\]
  By the definition of the support of a measure, $U\setminus \supp(\mu)$ is the largest  $\mu$-negligible open subset of $U$. Therefore, 
\[
\{z\in U: f(z) \ne 0\} \subset U\setminus \supp(\mu). 
\]  It follows that $f|_{\mathrm{supp}(\mu) \cap U}  \equiv 0$. 
\end{proof}

\begin{lemma}\label{lem-mono}
Let $S\subset \Omega$ be any  sub-domain containing an accumulation point of $\supp(\mu)$.  If $f\in \Hol(S)$ satisfies $f(z) = 0$ for $\mu$-almost every $z\in S$, then $f \equiv 0$ on $S$.
\end{lemma}
\begin{proof}
Let $S\subset \Omega$ and $f\in \Hol(S)$ be as given in the lemma. Lemma \ref{lem-supp-baohan} then implies that  $f|_{\mathrm{supp}(\mu) \cap S} \equiv 0$.  By assumption,  $\mathrm{supp}(\mu) \cap S$ contains an accumulation point in $S$, hence the Identity Theorem for holomorphic functions implies that  $f\equiv 0$ on $S$. 
\end{proof}

For any  open subset $U\subset \C$, let  $H^\infty(U)$ denote the space of bounded holomorphic functions on $U$. Then $H^\infty(U)$ is a Banach space equipped with the supremum norm. Clearly, for  any non-empty relatively compact subset $U\subset \Omega$,  the restriction map 
\[
R_U: A^p(\Omega, \mu)\rightarrow H^\infty(U)
\] given by  $R_U(f)= f|_U$ is well-defined.

\begin{lemma}\label{lem-S}
If $A^p(\Omega, \mu)$ is a Banach space, then for any  relatively compact sub-domain $S\subset \Omega$ containing an accumulation point of $\supp(\mu)$, the restriction map
\[
R_S: A^p(\Omega,\mu)\rightarrow H^\infty(S)
\]
  is continuous.
\end{lemma}

\begin{proof}
Fix a  relatively compact sub-domain $S\subset \Omega$ containing an accumulation point of $\supp(\mu)$.
By the Closed Graph Theorem, it suffices to show that the linear map $R_S$  has a closed graph. Take any sequence
\[\{ (f_n, f_n|_S)\}_{n=1}^\infty \subset A^p(\Omega, \mu)\times H^\infty(S)
\]   converging to $(f, g)\in A^p(\Omega, \mu)\times H^\infty(S)$  with repect to the product topology. That is,
\[
\lim_{n\to\infty}\| f_n - f\|_{A^p(\Omega, \mu)} =0 \an \lim_{n\to\infty}  \| f_n|_S - g\|_{H^\infty(S)}=0.
\]
Since $S$ is relatively compact, $\mu(S)<\infty$. It follows that
\[
\lim_{n\to\infty}\int_S | f_n(z) - f(z)|^p d\mu(z) =0 \an  \lim_{n\to\infty}\int_S | f_n(z) - g(z)|^p d\mu(z)  =0.
\]
Hence $
f(z)- g(z)  =0$ for  $\mu$-almost every $z\in S$.  By Lemma \ref{lem-mono}, $f - g \equiv 0$ on $S$. That is,
$
g = f|_S = R_S(f).
$
Thus the graph of  the map $R_S$ is closed and the proof of the lemma is completed.
\end{proof}

\begin{proof}[Proof of Theorem \ref{thm-main}]
Suppose that $A^p(\Omega, \mu)$ is Banach space and let  $K\subset \Omega$ be any compact subset.  Since $A^p(\Omega, \mu)$ is Banach space,  the semi-norm \eqref{def-norm} actually is a norm. Then $\Omega$ must contain an accumulation point, say $z_0\in \Omega$, of the set $\supp(\mu)$.  Let $\delta>0$ be small enough such that $D(z_0, \delta)\subset D(z_0, 2 \delta) \subset \Omega$. Since $\Omega$ is a domain, there exists a relatively compact sub-domain $S\subset \Omega$ with
\[
(D(z_0, \delta) \cup K) \subset  S.
\]
Since $S$ satisfies all assumptions of Lemma  \ref{lem-S} and $A^p(\Omega, \mu)$ is assumed to be a Banach space, by Lemma  \ref{lem-S}, the restriction map $R_S: A^p(\Omega, \mu) \rightarrow H^\infty(S)$ is  continuous. In other words, there exists a constant $\gamma_S>0$ such that
\[
\sup_{z\in K}|f(z)| \le \sup_{z\in S}|f(z)| \le \gamma_S \| f\|_{A^p(\Omega, \mu)}, \, \forall f \in A^p(\Omega, \mu).
\]
This completes the whole proof of the theorem.
\end{proof}


\begin{thebibliography}{KPR84}

\bibitem[AB02]{Dominating-sets}
Nicola Arcozzi and Anders Bj\"{o}rn.
\newblock Dominating sets for analytic and harmonic functions and completeness
  of weighted {B}ergman spaces.
\newblock {\em Math. Proc. R. Ir. Acad.}, 102A(2):175--192, 2002.

\bibitem[ARS09]{ASS}
Alexandru Aleman, Stefan Richter, and Carl Sundberg.
\newblock Nontangential limits in {$ P^t(\mu)$}-spaces and the index of
  invariant subspaces.
\newblock {\em Ann. of Math. (2)}, 169(2):449--490, 2009.

\bibitem[Bra55]{Bram}
Joseph Bram.
\newblock Subnormal operators.
\newblock {\em Duke Math. J.}, 22:75--94, 1955.

\bibitem[Con91]{conway}
John~B. Conway.
\newblock {\em The theory of subnormal operators}, volume~36 of {\em
  Mathematical Surveys and Monographs}.
\newblock American Mathematical Society, Providence, RI, 1991.

\bibitem[Hal82]{halmos}
Paul~R. Halmos.
\newblock {\em A {H}ilbert space problem book}, volume~17 of {\em Encyclopedia
  of Mathematics and its Applications}.
\newblock Springer-Verlag, New York-Berlin, second edition, 1982.

\bibitem[KPR84]{kalton}
Nigel~J. Kalton, Newton~T. Peck, and James~W. Roberts.
\newblock {\em An {$F$}-space sampler}, volume~89 of {\em London Mathematical
  Society Lecture Note Series}.
\newblock Cambridge University Press, Cambridge, 1984.

\bibitem[PR16]{paulsen}
Vern~I. Paulsen and Mrinal Raghupathi.
\newblock {\em An introduction to the theory of reproducing kernel {H}ilbert
  spaces}, volume 152 of {\em Cambridge Studies in Advanced Mathematics}.
\newblock Cambridge University Press, Cambridge, 2016.

\bibitem[Rud87]{Real-complex-analysis}
Walter Rudin.
\newblock {\em Real and complex analysis}.
\newblock McGraw-Hill Book Co., New York, third edition, 1987.

\bibitem[Tho91]{thomson}
James~E. Thomson.
\newblock Approximation in the mean by polynomials.
\newblock {\em Ann. of Math. (2)}, 133(3):477--507, 1991.

\end{thebibliography}

\end{document}